\documentclass[a4paper]{amsart}%

\usepackage{amsmath,amssymb,url}
\usepackage{graphics} 

\usepackage{float}
\usepackage[toc,page]{appendix}
\usepackage{enumerate}
\usepackage[english]{babel}
\usepackage{amscd}
\usepackage{graphicx}
\usepackage{array}
\usepackage{amscd}
\usepackage{cmll}
\usepackage[export]{adjustbox}
\usepackage{subcaption}
\usepackage{ae,aecompl} 
\usepackage{color}
\usepackage{hyperref}
\usepackage{enumitem}
\setlist{font=\normalfont}
\usepackage{hyphenat}

\numberwithin{equation}{section}

\theoremstyle{plain}
\newtheorem{theorem}{Theorem}[section]
\newtheorem{lemma}[theorem]{Lemma}

\newtheorem{corollary}[theorem]{Corollary}

\theoremstyle{definition}
\newtheorem{definition}[theorem]{Definition}

\newtheorem{remark}[theorem]{Remark}
\newtheorem{example}[theorem]{Example}
\newtheorem*{sdcthm}{Property~(SDC)}

\newcommand{\A}{\mathbb{A}}
\newcommand{\M}{\mathbb{M}}
\newcommand{\N}{\mathbb{N}}
\renewcommand{\L}{\mathbb{L}}
\newcommand{\B}{\mathbb{B}}
\newcommand{\E}{\mathcal{E}}
\DeclareMathOperator{\Sol}{Sol}
\DeclareMathOperator{\Eq}{Eq}

\renewcommand{\a}{\mathbf{a}}
\renewcommand{\O}{\mathcal{O}}
\newcommand{\OA}{\mathcal{O}_A}
\newcommand{\RA}{\mathcal{R}_A}
\newcommand{\Pol}{\operatorname{Pol}}
\newcommand{\Inv}{\operatorname{Inv}}
\newcommand{\rel}{\operatorname{rel}}
\DeclareMathOperator{\sdc}{(SDC)}

\begin{document}
\title[Solution sets of systems of equations over finite (semi)lattices]{Solution sets of systems of equations over finite lattices and semilattices}

\author[E. T\'{o}th]{Endre T\'{o}th}
\address[E. T\'{o}th]{Bolyai Institute, University of Szeged, Aradi v\'{e}rtan\'{u}k tere 1, H--6720 Szeged, Hungary}
\email{tothendre@math.u-szeged.hu}
\author[T. Waldhauser]{Tam\'as Waldhauser}
\address[T. Waldhauser]{Bolyai Institute, University of Szeged, Aradi v\'{e}rtan\'{u}k tere 1, H--6720 Szeged, Hungary}
\email{twaldha@math.u-szeged.hu}



\begin{abstract}
Solution sets of systems of homogeneous linear equations over fields are characterized as being subspaces, i.e., sets that are closed under linear combinations. Our goal is to characterize solution sets of systems of equations over arbitrary finite algebras by a similar closure condition. We show that solution sets are always closed under the centralizer of the clone of term operations of the given algebra; moreover, the centralizer is the only clone that could characterize solution sets. If every centralizer-closed set is the set of all solutions of a system of equations over a finite algebra, then we say that the algebra has Property~$\sdc$. Our main result is the description of finite lattices and semilattices with Property~$\sdc$: we prove that a finite lattice has Property~$\sdc$ if and only if it is a Boolean lattice, and a finite semilattice has Property~$\sdc$ if and only if it is distributive.
\end{abstract}

\maketitle

\section{Introduction}
In universal algebra, investigations of systems of equations usually focus on either finding a solution, the complexity of finding a solution or deciding if there is a solution at all. For us the main interest is the ``shape" of the solution sets, just like in the following basic result of linear algebra: solution sets of systems of
homogeneous linear equations in $n$ variables over a field $K$ are precisely the subspaces of the vector space $K^{n}$, i.e., sets of $n$-tuples that are
closed under linear combinations. Our goal is to give a similar characterization (i.e., a kind of closure condition) for solution sets of systems of equations over arbitrary finite algebras.

Let us fix a nonempty set $A$ and a set $F$ of operations on $A$; then we obtain the algebra $\A=(A, F)$. Any equation over $\A$ is of the form $f(x_1, \dots, x_n) = g(x_1, \dots, x_n)$, where $f$ and $g$ are $n$-ary term functions. We can also say that $f$ and $g$ are from the set $[F]$ of operations generated by $F$ by means of compositions. After this observation we can see that in every equation, the operations on both sides are from $C:=[F]$, which we will call the \emph{clone} generated by $F$ (Definition~\ref{clone}). We will investigate solution sets of systems of equations over finite algebras in this view. The algebraic sets studied by B.~I.~Plotkin in his universal algebraic geometry \cite{Plot} are essentially the same as our solution sets; the only difference being that we consider only finite systems of equations. Recently A.~Di~Nola, G.~Lenzi and G.~Vitale characterized the solution sets of certain systems of equations over lattice ordered abelian groups (see \cite{ANGLGV}).

 In our previous paper \cite{ETTW} we proved that for any system of equations over a clone $C$, the solution set is closed under the \emph{centralizer} of the clone $C$ (see Definition~\ref{centralizer}). We also proved that for clones of Boolean functions this condition is sufficient as well. We will say that a clone (or the associated algebra) has Property~$\sdc$ if closure under the centralizer characterizes the solution sets (here SDC stands for ``Solution sets are Definable by closure under the Centralizer"). Thus clones of Boolean functions (i.e., two-element algebras) always have Property~$\sdc$, and in \cite{ETTW} we gave an example of a three-element algebra that does not have Property~$\sdc$. In this paper we describe all finite lattices and semilattices with Property~$\sdc$.
In Section~\ref{prel} we present the necessary notations and definitions. In Section~\ref{quant} we give a connection between Property~$\sdc$ and quantifier elimination of certain primitive positive formulas. Also we show that for systems of equations over a clone $C$, if all solution sets can be described by closure under a clone $D$, then $D$ must be the centralizer of $C$. Section~\ref{latt} contains the full description of finite lattices with Property~$\sdc$: a finite lattice has Property~$\sdc$ if and only if it is a Boolean lattice. In Section~\ref{semil} finite semilattices having Property~$\sdc$ are described as semilattice reducts of distributive lattices.

\section{Preliminaries\label{prel}}

\subsection{Operations and clones} \label{soper}

Let $A$ be an arbitrary set with at least two elements. By an\textbf{\ }
\emph{operation} on $A$ we mean a map $f\colon A^{n}\rightarrow A$; the positive integer $n$ is called the \emph{arity} of the operation $f$. The set of all operations on $A$ is denoted by $\OA$. For a set $F\subseteq\OA$ of operations, by $F^{(n)}$ we mean the set of $n$-ary members of $F$. In particular, $\OA^{(n)}$ stands for the set of all $n$-ary operations on $A$. 

We will denote tuples by boldface letters, and we will use the corresponding plain letters with subscripts for the components of the tuples. For example,
if $\mathbf{a}\in A^{n}$, then $a_{i}$ denotes the $i$-th component of $\mathbf{a}$, i.e., $\mathbf{a}=(a_{1},\dots,a_{n})$. In particular, if
$f\in\OA^{(n)}$, then $f(\mathbf{a})$ is a short form for $f(a_{1},\dots,a_{n})$. If $\mathbf{t}^{(1)},\dots,\mathbf{t}^{(m)}\in A^{n}$ and $f\in\mathcal{O}_{A}^{(m)}$, then $f(\mathbf{t}^{(1)},\dots,\mathbf{t}^{(m)})$ denotes the $n$-tuple obtained by applying $f$ to the tuples $\mathbf{t}^{(1)},\dots,\mathbf{t}^{(m)}$ componentwise:
\[
f(\mathbf{t}^{(1)},\dots,\mathbf{t}^{(m)})=\bigl(f(t_{1}^{(1)},\dots
,t_{1}^{(m)}),\dots,f(t_{n}^{(1)},\dots,t_{n}^{(m)})\bigr).
\]
We say that $T\subseteq A^{n}$ is \emph{closed under $C$}, if for all
$m\in\mathbb{N},\mathbf{t}^{(1)},\dots,\mathbf{t}^{(m)}\in T$ and for all
$f\in C^{(m)}$ we have $f(\mathbf{t}^{(1)},\dots,\mathbf{t}^{(m)})\in T$.

Let $f\in\OA^{(n)}$ and $g_{1},\dots,g_{n}\in\OA^{(k)}$. By the \emph{composition} of $f$ by $g_{1},\dots,g_{n}$ we mean the
operation $h\in\OA^{(k)}$ defined by
\[
h(\mathbf{x})=f\bigl(g_{1}(\mathbf{x}),\dots,g_{n}(\mathbf{x})\bigr)\text{ for
all }\mathbf{x}\in A^{k}.
\]
Now we present the precise definition of  clones.
\begin{definition}\label{clone}
	If $C\subseteq\OA$ is closed under composition and contains the \emph{projections} $e_i^{(n)}\colon (x_{1},\dots,x_{n})\mapsto x_{i}$ for all $1\leq i\leq
n\in\mathbb{N}$, then $C$ is said to be a \emph{clone} (notation:\ $C\leq
\OA$).
\end{definition}

For an arbitrary set $F$ of operations on $A$, there is a least clone $[F]$ containing $F$, called the clone \emph{generated} by $F$. The elements of this clone are those operations that can be obtained from members of $F$ and from projections by finitely many compositions. In other words, $[F]$ is the set of term operations of the algebra $\A = (A,F)$.

The set of all clones on $A$ is a lattice under inclusion; the greatest element of this lattice is $\OA$, and the least element is the \emph{trivial clone} consisting of projections only. There are countably infinitely many clones on the two-element set; these have been described by Post \cite{Post}, hence the lattice of clones on $\{0,1\}$ is called the \emph{Post lattice}. If $A$ is a finite set with at least three elements, then the clone lattice on $A$ is of continuum cardinality \cite{JanMuc}, and it is a very difficult open problem to describe all clones on $A$ even for $\vert A\vert =3$.

\subsection{Centralizer clones}

We say that the operations $f\in\OA^{(n)}$ and $g\in
\OA^{(m)}$ \emph{commute }(notation: $f\perp g$) if
\begin{multline*}
f\bigl(g(a_{11},a_{12},\dotsc,a_{1m}),\dotsc,g(a_{n1},a_{n2},\dotsc
,a_{nm})\bigr)\\
=g\bigl(f(a_{11},a_{21},\dotsc,a_{n1}),\dotsc,f(a_{1m},a_{2m},\dotsc
,a_{nm})\bigr)
\end{multline*}
holds for all $a_{ij}\in A~(1\leq i\leq n,1\leq j\leq m)$. This can be visualized as follows: for every $n\times m$ matrix $Q=(a_{ij})$, first applying $g$ to the rows of $Q$ and then applying $f$ to the resulting column vector yields the same result as first applying $f$ to the columns of $Q$ and then applying $g$ to the resulting row vector (see Figure~\ref{comm}).

\begin{figure}
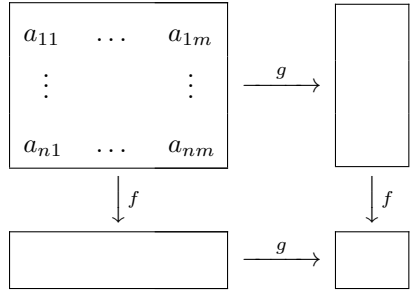

\[
\begin{CD} \renewcommand{\arraystretch}{1.6}\begin{tabular}{|m{0.6cm}m{0.6cm}m{0.6cm}|} \cline{1-3}\cline{3-3} $a_{11}$ & $\dots$ & $a_{1m}$ \\ \hfil$\vdots$ & & \hfil$\vdots$\\ $a_{n1}$ & $\dots$ & $a_{nm}$ \\ \cline{1-3}\cline{3-3} \end{tabular} @>g>> \renewcommand{\arraystretch}{1.66}\begin{tabular}{m{0.6cm}} \cline{1-1} \multicolumn{1}{|m{0.6cm}|}{} \\ \multicolumn{1}{|c|}{} \\ \multicolumn{1}{|c|}{} \\ \cline{1-1} \end{tabular} \\ @VVfV @VVfV\\ \renewcommand{\arraystretch}{1.66}\begin{tabular}{|m{0.6cm}m{0.6cm}m{0.6cm}|} \cline{1-3}\cline{3-3} & & \\ \cline{1-3}\cline{3-3} \end{tabular} @>g>> \renewcommand{\arraystretch}{1.66}\begin{tabular}{|m{0.6cm}|} \cline{1-1} \\ \cline{1-1} \end{tabular} \end{CD}
\]
\caption{Commutation of $f$ and $g$.} \label{comm}
\end{figure}

\begin{definition} \label{centralizer}
For any $F\subseteq\OA$, the set $F^{\ast}:=\{g\in\OA\mid f\perp g$ for all $f\in F\}$ is called the \emph{centralizer} of $F$.
\end{definition} 
It is easy to verify that if $f,g_{1},\dots,g_{n}$ all commute with an operation $h$, then the composition $f(g_{1},\dots,g_{n})$ also commutes with $h$. This implies that $F^\ast$ is a clone for all $F \subseteq \OA$ (even if $F$ itself is not a clone).

Clones arising in this form are called \emph{primitive positive clones}; such clones seem to be quite rare: there are only finitely many primitive positive clones over any finite set \cite{BurrisWillard}.

\begin{example}
\label{example linear centralizers}Let $K$ be a field, and let $L$ be the
clone of all operations over $K$ that are represented by a linear polynomial:
\[
L:=\{a_{1}x_{1}+\dots+a_{k}x_{k}+c \mid k\geq0,a_{1},\dots,a_{k},c\in K\}.
\]
Since $L$ is generated by the operations $x+y$, $ax~(a\in K)$ and the
constants $c\in K$, the centralizer $L^{\ast}$ consists of those operations
$f$ over $K$ that commute with $x+y$ and $ax$ (i.e., $f$ is additive and
homogeneous), and also commute with the constants (i.e., $f(c,\dots,c)=c$ for
all $c\in K$):
\[
L^{\ast}:=\{a_{1}x_{1}+\dots+a_{k}x_{k} \mid k\geq1,a_{1},\dots,a_{k}\in
K\text{ and }a_{1}+\dots+a_{k}=1\}.
\]
Similarly, one can verify that $L_{0}^{\ast}=L_{0}$ for the clone
\[
L_{0}:=\{a_{1}x_{1}+\dots+a_{k}x_{k} \mid k\geq0,a_{1},\dots,a_{k}\in K\}.
\]

\end{example}

\subsection{Equations and solution sets\label{subsect Eq and Sol}}

Let us fix a finite set $A$, a clone $C\leq\OA$ and a natural number $n$. By an
$n$-ary \emph{equation over $C$ }(\emph{$C$-equation} for short) we mean an
equation of the form $f(x_{1},\dots,x_{n})=g(x_{1},\dots,x_{n})$, where
$f,g\in C^{(n)}$. We will often simply write this equation as a pair $(f,g)$.
A \emph{system of $C$-equations} is a finite set of $C$-equations of the same
arity:
\[
\mathcal{E}:=\bigl\{(f_{1},g_{1}),\dots,(f_{t},g_{t})\bigr\}\text{, where
}f_{i},g_{i}\in C^{(n)}\text{ }(i=1,\dots,t).
\]
Note that we consider only systems consisting of a finite number of equations. This does not restrict generality, since we are dealing only with finite algebras.
We define the \emph{set of solutions of $\mathcal{E}$} as the set
\[
\operatorname{Sol}(\mathcal{E}):=\bigl\{\mathbf{a}\in A^{n} \mid 
f_{i}(\mathbf{a})=g_{i}(\mathbf{a})\text{ for }i=1,\dots,t\bigr\}.
\]
For $\mathbf{a}\in A^{n}$ we denote by $\operatorname{Eq}_{C}(\mathbf{a})$ the
set of $C$-equations satisfied by $\mathbf{a}$:
\[
\operatorname{Eq}_{C}(\mathbf{a}):=\bigl\{(f,g) \mid f,g\in C^{(n)}\text{
and }f(\mathbf{a})=g(\mathbf{a})\bigr\}.
\]
Let $T\subseteq A^{n}$ be an arbitrary set of tuples. We denote by
$\operatorname{Eq}_{C}(T)$ the set of $C$-equations satisfied by $T$:
\[
\operatorname{Eq}_{C}(T):=\bigcap_{\mathbf{a}\in T}\operatorname{Eq}_{C}(\mathbf{a}).
\]

\begin{remark}\label{soleqt}
For any given $n \in \N$ and $C\leq \OA$, the operators $\Sol$ and $\Eq_C$ give rise to a Galois connection between sets of $n$-tuples and systems of $n$-ary equations. In particular, if $T$ is the solution set of a system of equations (i.e., $T$ is Galois closed), then $T=\Sol(\Eq_C(T))$; moreover, $\E=\Eq_C(T)$ is the largest system of equations with $T=\Sol(\E)$.
\end{remark}

\begin{example}
\label{example linear equations}Considering the \textquotedblleft
linear\textquotedblright\ clones of Example~\ref{example linear centralizers},
$L$-equations are linear equations and $L_{0}$-equations are homogeneous
linear equations.
\end{example}

In a previous paper \cite{ETTW} we proved that for any clone, the solution sets are closed under the centralizer of the clone. Furthermore, we proved the following theorem, which characterizes solution sets of systems of equations over clones of Boolean functions.

\begin{theorem}[\cite{ETTW}]\label{fotetel}
	For any clone of Boolean functions $C\leq \mathcal{O}_{\{0,1\}}$ and $T\subseteq \{0,1\}^n$, the following conditions are equivalent:
	\begin{enumerate}[label=(\alph*)]
		\item there is a system $\E$ of $C$-equations such that $T=\Sol(\E)$;
		\item $T$ is closed under $C^\ast$.
	\end{enumerate}
\end{theorem}

\noindent Thus for two-element algebras, closure under the centralizer characterizes solution sets. We will say that a clone $C$ has Property~$\sdc$, if this is true for the clone:

\begin{sdcthm} \label{cond}
	The following are equivalent for all $n \in \N$ and $T \subseteq A^n$:
	\begin{enumerate}[label=(\alph*)]
		\item there exists a system $\E$ of $C$-equations such that $T=\Sol(\E)$;
		\item the set $T$ is closed under $C^\ast$.
	\end{enumerate}
\end{sdcthm}
\noindent Here SDC is an abbreviation for ``Solution sets are Definable by closure under the Centralizer''.
In \cite{ETTW} we presented a clone on a three-element set that does not have Property~$\sdc$, showing that in general this is not a trivial property.

\subsection{The Pol-Inv Galois connection}

	For a positive integer $h$, a set $\rho\subseteq A^{h}$ is called an \emph{$h$-ary relation} on $A$; let $\RA$ denote the set of all relations on $A$. For any $R \subseteq \RA$, let $R^{(h)}$ denote the \emph{$h$-ary part of $R$}, i.e., the set of $h$-ary members of $R$.
	
	For a relation $\rho \subseteq A^h$ and operation $f \in \OA^{(n)}$, if for arbitrary tuples $\a^{(1)}, \dots, \a^{(n)} \in \rho$ we have $f(\a^{(1)}, \dots, \a^{(n)}) \in \rho$, then we say that $f$ is a \emph{polymorphism} of $\rho$, or $\rho$ is an \emph{invariant relation} of $f$ (or we also say that $f$ \emph{preserves} $\rho$). We will denote this as $f\vartriangleright\rho$. Note that $f \vartriangleright \rho$ is equivalent to $\rho$ being closed under $f$ (see Subsection~\ref{soper}). Preservation induces the so-called \emph{$\Pol$-$\Inv$ Galois connection}. For any $F \subseteq\OA$ and for any $R \subseteq\mathcal{R}_A$, let
	\begin{align*}
		\Inv{(F)} & :=\{ \rho \in\RA\mid\forall f\in F\colon f\vartriangleright\rho\}, \text{ and} \\
		\Pol{(R)} & :=\{ f \in \OA\mid\forall\rho\in R\colon f\vartriangleright\rho\}.
	\end{align*}

It is easy to verify that $\Pol{(R)}$ is a clone for all $R \subseteq \RA$. Moreover, for every set of operations $F$ on a finite set, the clone generated by $F$ is $[F]=\Pol(\Inv(F))$ by the results of Bodnar\v{c}uk, Kalu\v{z}nin, Kotov, Romov and Geiger \cite{BodKalKotRom, Gei}.

	Given a set of relations $R \subseteq \RA$, a \emph{primitive positive formula over $R$} (pp. formula for short)  is a formula of the form
	\begin{equation}
		\Phi(  x_{1},\dots,x_{n})  =\exists y_{1}\exists y_{2}\dots\exists y_{m}\bigwith_{j=1}^{t}\rho_{j}\bigl(z_{1}^{(  j)  },				\dots,z_{r_{j}}^{(  j)  }\bigr) \label{eq pp},
	\end{equation}
	where $\rho_j \in R^{(r_j)}$, and $z_i^{(j)} \ (j = 1, \dots, t, \text{ and } i=1, \dots, r_j)$ are variables from the set $\{x_1, \dots, x_n, y_1, \dots, y_m\}$.
	The set 
\[
\rel{(\Phi)}:=\{ (a_1, \dots, a_n) \mid\Phi(a_1, \dots , a_n) \text{ is true}\}
\]
 is an $n$-ary relation, which is \emph{the relation defined by $\Phi$}.
	If $R \subseteq \RA$, then let $\langle R\rangle_\exists$ denote the set of all relations that can be defined by a primitive positive formula over $R\cup\{=\}$, and let $\langle R \rangle_\nexists$ denote the set of all relations that can be defined by a quantifier-free primitive positive formula over $R\cup\{=\}$. If $R \subseteq \RA$ contains the equality relation and $R$ is closed under primitive positive definability, then we say that $R$ is a \emph{relational clone}. The relational clone generated by $R$ is $\langle R \rangle_\exists=\Inv(\Pol(R))$ \cite{BodKalKotRom, Gei}.

For $f \in\OA^{(n)}$, we define the following relation on $A$, called the \emph{graph of} $f$:	
\[
f^{\bullet} =\{  (  a_{1},\dots,a_{n},b)  \mid f(a_{1},\dots,a_{n})  =b\}  \subseteq A^{n+1}.
\]
For $F\subseteq\OA$, let $F^{\bullet} =\{  f^{\bullet} \mid f\in F\}$. It is not hard to see that for any $f \in \OA^{(n)}$ and $g \in \OA^{(m)}$ the function $f$ commutes with $g$ if and only if $f$ preserves the graph of $g$ (or equivalently, if and only if $g$ preserves the graph of $f$). Therefore for any $F \subseteq \OA$ we have $F^\ast=\Pol(F^\bullet)$.

\section{Quantifier elimination} \label{quant}

Let $F \subseteq \OA$, then let $F^{\circ} $ denote the set of all relations that are solution sets of some equation over $F$:	
\[
F^{\circ} =\bigl\{ \Sol(f,g) \bigm| n\in\mathbb{N},\mathbb{~}f,g\in F^{(n)} \bigr\} \subseteq \RA.
\]
The following remark shows that the graph of an operation $f \in F$ also belongs to $F^\circ$.

\begin{remark} \label{rem2}
	Let $f \in \OA^{(n)}$, and define $\widetilde{f} \in \OA^{(n+1)}$ by $\widetilde{f}(x_1,\dots, x_n, x_{n+1}):=f(x_1,\dots, x_n)$. Then we have
	\begin{align*}
	\Sol\big(\widetilde{f}, e_{n+1}^{(n+1)}\big) & =\bigl\{ (a_1, \dots, a_n, b) \in A^{n+1} \bigm| \widetilde{f}(a_1, \dots, a_n, b)=e_{n+1}^{(n+1)}(a_1,\dots, a_n, b) \bigr\} \\
	& =\bigl\{ (a_1, \dots, a_n, b) \in A^{n+1} \bigm| f(a_1, \dots, a_n)=b \bigr\}= f^\bullet.
	\end{align*}
\end{remark}

The following three lemmas prepare the proof of Theorem~\ref{equiv}, which gives us an equivalent condition to Property~$\sdc$ that we will use in sections \ref{latt} and \ref{semil}.

\begin{lemma}\label{pontok}
	For every clone $C\leq\OA$, we have $C^{\bullet}\subseteq C^{\circ}$ and $\left\langle C^{\bullet}\right\rangle _{\exists}=\left\langle C^{\circ}\right\rangle _{\exists}$.
\end{lemma}
\begin{proof}
	In accordance with Remark~\ref{rem2}, for all $f \in C$ we have $\Sol(\widetilde{f},e_{n+1}^{(n+1)})=f^\bullet \in C^\circ$. Therefore $C^\bullet 			\subseteq C^\circ$, which implies that $ \langle C^\bullet \rangle_\exists \subseteq \langle C^\circ \rangle_\exists$. To prove the reversed containment, 		let us consider an arbitrary relation $\rho=\Sol(f,g) \in C^\circ$ with $f,g \in C^{(n)}$. Then, for any $(x_1, \dots, x_n) \in A^n$, we have
	\begin{align*}
	(x_1, \dots, x_n) \in  \rho & \iff f(x_1, \dots, x_n)=g(x_1, \dots, x_n)\\
	& \iff \exists y\colon f(x_1,\dots, x_n)=y \with g(x_1, \dots, x_n)=y \\
	& \iff \exists y\colon (x_1, \dots, x_n, y) \in f^\bullet \with (x_1, \dots, x_n, y) \in g^\bullet.
	\end{align*}
	This means that $\rho$ can be defined by a pp. formula over $\{f^\bullet, g^\bullet\}$, therefore $\rho \in \langle C^\bullet \rangle_\exists$. Thus, we 			obtain  $C^\circ \subseteq \langle C^\bullet \rangle_\exists$, and this implies that $\langle C^\circ \rangle_\exists \subseteq \langle \langle C^\bullet 			\rangle_\exists \rangle_\exists = \langle C^\bullet \rangle_\exists$. 
	Therefore $\langle C^{\bullet}\rangle _\exists=\langle C^{\circ}\rangle _\exists$.
\end{proof}

\begin{lemma}\label{kvmsystem}
	For every clone $C\leq\OA$ and $T\subseteq A^{n}$, there is a system $\mathcal{E}$ of $C$-equations such that $T=\Sol(\mathcal{E})$ if and only if $T\in\langle C^{\circ}\rangle _\nexists$.
\end{lemma}
\begin{proof}
	Let $\Phi$ be an arbitrary quantifier-free pp. formula over $C^\circ$. By definition, $\Phi$ is of the form 
\[
\Phi(x_1, \dots, x_n)=\bigwith_{j=1}^{t} \Sol(f_j, g_j)=\bigwith_{j=1}^{t} \Big( f_j\big(z_1^{(j)}, \dots, z_{r_j}^{(j)}\big)=g_j\big(z_1^{(j)}, \dots, z_{r_j}^{(j)}\big) \Big),
\]
	where $n, t \in \N, f_j, g_j \in C^{(r_j)}$ and $z_1^{(j)}, \dots, z_{r_j}^{(j)} \in \{x_1, \dots, x_n\}$ for all $j= 1, \dots, t$. We define the operations $\widetilde{f_j}(x_1, \dots, x_n):=f_j(z_1^{(j)}, \dots, z_{r_j}^{(j)})$ and $\widetilde{g_j}(x_1, \dots, x_n):=g_j(z_1^{(j)}, \dots, z_{r_j}^{(j)})$ (by identifying variables and by adding fictitious variables) for all $j=1, \dots, t$.
	Then $\Phi$ is equivalent to the formula
\[
\Psi(x_1, \dots, x_n)=\bigwith_{j=1}^{t} \Big( \widetilde{f_j}\big(x_1, \dots, x_n\big)=\widetilde{g_j}\big(x_1, \dots, x_n\big) \Big),
\]
	and $\widetilde{f_j}, \widetilde{g_j} \in C^{(n)}$ for all $j= 1, \dots, t$. Since $\Phi$ and $\Psi$ are equivalent, they define the same set $T \subseteq A^n$, and it is obvious that the set defined by $\Psi$ is the solution set of the system $\{(\widetilde{f_1},\widetilde{g_1}),\dots,(\widetilde{f_t},\widetilde{g_t})\}$. Conversely, it is clear that every solution set can be defined by a quantifier-free pp. formula of the form of $\Psi$.
\end {proof}

\begin{lemma} \label{nec}
	For every clone $C \leq \OA$, we have $\Inv{(C^\ast)}=\langle C^\bullet \rangle_\exists$. Consequently, a set $T \subseteq A^n$ is closed under $C^		\ast$ if and only if $T\in\langle C^\circ \rangle _\exists$.
\end{lemma}
\begin{proof}
	From Section~\ref{prel} using that $F^\ast=\Pol(F^\bullet)$ and that $\Inv(\Pol(R))= \langle R \rangle_\exists$, we have
\[
\Inv{(C^\ast)} =\Inv{(\Pol{(C^\bullet)})}=\langle C^\bullet \rangle_\exists.
\]
	\noindent The second statement of the lemma follows immediately from Lemma~\ref{pontok} by observing that $T$ is closed under $C^\ast$ if and only if 
	$T \in \Inv(C^\ast)$.
\end{proof}

\begin{theorem}[\cite{ETTW}] \label{ttetel}
	For every clone $C\leq\OA$ and $T\subseteq A^{n}$, if there is a system $\mathcal{E}$ of $C$-equations such that $T=					\Sol(\mathcal{E})  $, then $T$ is closed under $C^{\ast}$.
\end{theorem}
\begin{proof}
	Let $C\leq\OA, T\subseteq A^{n}$, and let $\mathcal{E}$ be a system of $C$-equations and $T=\Sol(\mathcal{E})$. By Lemma~\ref{kvmsystem} we 	have  $T \in \langle C^\circ \rangle_\nexists \subseteq \langle C^\circ \rangle_\exists$. Using Lemma~\ref{nec}, this means that $T$ is closed under $C^\ast$.
\end{proof}

	The previous theorem shows that in Property~$\sdc$, condition (a) implies (b). Therefore, for all clones $C \leq \OA$, it suffices to investigate the implication $(\text{b}) \implies (\text{a})$. As a consequence of lemmas \ref{pontok}, \ref{kvmsystem} and \ref{nec}, we obtain the promised equivalent reformulation of Property~$\sdc$ in terms of quantifier elimination.

\begin{theorem} \label{equiv}
	For every clone $C\leq\OA$, the following five conditions are equivalent:
	\begin{enumerate}[label=(\roman*)]
		\item $C$ has Property~$\sdc$;
		\item $\langle C^\circ\rangle _\nexists=\Inv{(C^{\ast})}$;
		\item $\langle C^\circ\rangle_\nexists=\langle C^\circ\rangle_\exists$;
		\item every primitive positive formula over $C^{\circ}$ is equivalent to a quantifier-free primitive positive formula over $C^\circ$;
		\item $\langle C^\circ\rangle _\nexists$ is a relational clone.
	\end{enumerate}
\end{theorem}
\begin{proof}
	(i)$\iff$(ii): By Lemma~\ref{kvmsystem}, $T$ is the solution set of some system of equations over $C$ if and only if $T \in \langle C^\circ \rangle_\nexists$.
	
	(ii)$\iff$(iii): This follows from (the proof of) Lemma~\ref{nec}.
	
	(iii)$\iff$(iv): This is trivial.

	(iii)$\iff$(v): This follows from the fact that the relational clone generated by $\langle C^\circ\rangle_\nexists$ is $\langle C^\circ\rangle_\exists$.
\end{proof}

In the following corollary we will see that Theorem~\ref{equiv} implies that $C^\ast$ is the only clone that can describe solution sets over $C$ (if there is such a clone at all). Thus, the abbreviation SDC can also stand for ``Solution sets are Definable by closure under any Clone''.

\begin{corollary} \label{csakccsillag}
	Let $C \leq \OA$ be a clone, and assume that there is a clone $D$ such that for all $n\in \N$ and $T \subseteq A^n$ the following equivalence holds:
	\[
	T \text{ is the solution set of a system of $C$-equations} \iff T \text{ is closed under } D.
	\]
	Then we have $D=C^\ast$.
\end{corollary}
\begin{proof}
	 The condition in the corollary gives us by Lemma~\ref{kvmsystem} that for all $T \subseteq A^n$, we have $T \in \langle C^\circ \rangle_\nexists$ if and only if $T \in \Inv{(D)}$. This means that $\langle C^\circ \rangle_\nexists = \Inv{(D)}$, thus $\langle C^\circ\rangle _\nexists$ is a relational clone. Therefore, by  Theorem~\ref{equiv} this is equivalent to the condition $\Inv{(C^\ast)}= \langle C^\circ \rangle_\nexists =\Inv{(D)}$. Applying the operator $\Pol$ to the last equality we get that 
\[
C^\ast=\Pol{(\Inv{(C^\ast)})}=\Pol{(\Inv{(D)})}=D.
\]
\end{proof}

\section{Systems of equations over lattices}\label{latt}

	In this, and in the following section $\L=(L, \land, \lor)$ denotes a finite lattice, with meet operation $\land$ and join operation $\lor$. Furthermore, $0_\L$ 		denotes the least and $1_\L$ denotes the greatest element of $\L$ (that is, $0_\L=\bigwedge L$ and $1_\L=\bigvee L$).

The following lemma shows that Property~$\sdc$ does not hold for non-distributive lattices, i.e., solution sets of systems of equations over a non-distributive lattice can not be characterized via closure conditions.

\begin{lemma} \label{halo1}
	Let $\L=(L, \land, \lor)$ be a finite lattice. If Property~$\sdc$ holds for $C=[\land, \lor]$, then $\L$ is a distributive lattice.
\end{lemma}
\begin{proof}
	Let $\L=(L, \land, \lor)$ be a non-distributive finite lattice and $C=[\land, \lor] \leq \O_L$. By Lemma~\ref{nec}, the set 
\[
T=\{ (x,y) \mid \exists u \in L\colon u \land x = u \land y \text{ and } u \lor x = u \lor y \} \subseteq L^2
\]
	 is closed under $C^\ast$. We prove that $T$ is not the solution set of a system of equations over $C$, hence Property~$\sdc$ does not hold for 
	$C$. Suppose that there exists a system of $C$-equations $\E$ such that $T=\Sol(\E)$. Since $\L$ is not distributive, by Birkhoff's theorem we know that there is a sublattice of $\L$, which is isomorphic either to $N_5$ or $M_3$. Now neither of the equations 
\[
x=y \ (\iff x \land y = x \lor y),\quad x=x \land y,\quad x=x \lor y,\quad y = x \land y, \quad y = x \lor y
\]
	belong to $\E$; we prove this by presenting a counterexample for each equation. These counterexamples are shown in  Figure~\ref{fig2}, where we choose the elements $a$ and $b$ as presented in the figure. (Note that an element $u$, chosen like on the figure, shows that $(a,b), (b, a) \in T$. In the table, the entry $(x_1,y_1)$ in the line starting with the term $s(x,y)$ and column starting with the term $t(x,y)$ witnesses that $(x_1,y_1)$ is not a solution of $s(x, y) = t(x, y)$.)
\begin{figure}
		\begin{subfigure}{0.25\textwidth}
			\includegraphics[width=0.7\linewidth, center]{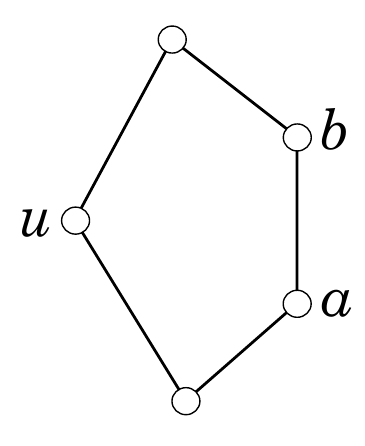}
		\end{subfigure}
		\begin{subfigure}{0.25\textwidth}
			\includegraphics[width=0.7\linewidth, center]{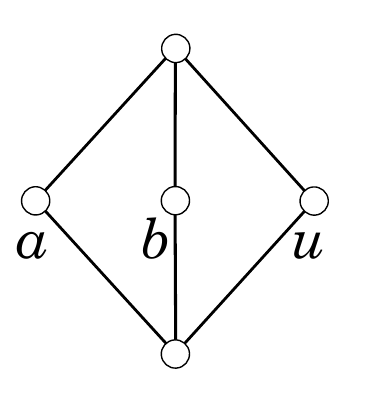}
		\end{subfigure}
		\begin{subfigure}{0.45\textwidth} \renewcommand{\arraystretch}{1.15}
		\begin{tabular}
			[c]{c|cccc}
			$=$ & $x$ & $y$ & $x \land y$ & $x \lor y$\\\hline
			$x$ & $-$ & $(a,b)$ & $(b,a)$ & $(a,b)$\\
			$y$ & $-$ & $-$ & $(a,b)$ & $(b,a)$\\
			$x \land y$ & $-$ & $-$ & $-$ & $(a,b)$\\
			$x \lor y$ & $-$ & $-$ & $-$ & $-$
		\end{tabular}
		\end{subfigure}
		\caption{Counterexamples showing that these equations do not belong to $\E$.}  \label{fig2}
\end{figure}
	There are no other non-trivial 2-variable equations over $C$, therefore we get that $T$ satifies only trivial equations, hence $T=L^2$. This is a contradiction, since $(0_\L, 1_\L) \notin T$.
\end{proof}

	The following lemma will help us prove that Property~$\sdc$ can only hold for Boolean lattices. Before the lemma, for a distributive lattice $\L$ we define the \emph{median} of the elements $x, y, z \in L$ as
\[
m(x, y, z) = (x \land y) \lor (x \land z) \lor (y \land z) = (x \lor y) \land (x \lor z) \land (y \lor z).
\]
\begin{lemma} \label{pq}
	Let $\L=(L, \land, \lor)$ be a distributive lattice, and for all $x, y, z, u \in L$ let 
\[
p(x,y,z,u)=(x \land y) \lor (x \land z) \lor (y \land z) \lor (u \land x) \lor (u \land y) \lor (u \land z).
\]
	Then for all $x, y, z, u \in L$ we have
\[
p(x,y,z,u)=x \lor y \lor z \lor u \iff m(x,y,z) \lor u = x \lor y \lor z.
\]
\end{lemma}
\begin{proof}
	 Let $x,y,z, u \in L$ be arbitrary elements. Let us denote $m(x,y,z)$ simply by $m$ and $p(x,y,z, u)$ by $p$ for better readability. 
	 
	 First let us suppose that $p=x \lor y \lor z \lor u$. It is easy to see that $p \leq x \lor y \lor z$ always holds (since every meet in $p$ is less than or equal to $x \lor y 	\lor z$). Since $p=x \lor y \lor z \lor u$, we get that  $p \leq x \lor y \lor z \leq x \lor y \lor z \lor u = p$, hence $p=x \lor y \lor z$. Observe that by the 			distributivity of $\L$, $p$ can be rewritten as $p=m \lor (u \land (x \lor y \lor z))$, and from the previous 
	chain of inequalities we can see that $u \leq x \lor y \lor z$, therefore we have $p=m \lor u$.  Thus $ m \lor u = p = x \lor y \lor z$.
	 
	  For the other direction suppose that $m \lor u = x \lor y \lor z$. Using that $\L$ is distributive, we get that $p = m \lor (u \land (x \lor y \lor 				z))= (m \lor u) \land (m \lor (x \lor y \lor z))$, and by the assumption this implies that $p=x \lor y \lor z$. Our assumption also implies that $u \leq x \lor y 			\lor z$, therefore we have $p = x \lor y \lor z \lor u$.
\end{proof}

\begin{theorem} \label{halo2}
	Let $\L=(L, \land, \lor)$ be a finite distributive lattice. If Property~$\sdc$ holds for $C=[\land, \lor]$, then $\L$ is a Boolean lattice.
\end{theorem}
\begin{proof}
	Let $\L=(L, \land, \lor)$ be a finite distributive lattice and let $C=[\land, \lor] \leq \O_L$.
	Since $\L$ is distributive, by Birkhoff's representation theorem $\L$ can be embedded into a Boolean lattice $\B$, hence we may suppose without loss of 			generality that $\L$ is already a sublattice of $\B$. We can also assume that $0_\L=0_\B$ and $1_\L=1_\B$. Let us denote the complement of an 				element $x \in \B$ by $x'$. We define the dual of $p=p(x,y,z,u)$ (from Lemma~\ref{pq}) as $p^d=q=q(x,y,z,u)$, i.e.,
\[
q(x,y,z,u)= (x \lor y) \land (x \lor z) \land (y \lor z) \land (u \lor x) \land (u \lor y) \land (u \lor z).
\]
	 Let $T$ be the following set:
	\begin{align*}
		T = \big\{(x,y,z) \in L^3 \bigm|  \exists u \in L\colon & p(x,y,z,u) = x \lor y \lor z \lor u \text{ and } \\
		& q(x,y,z,u) = x \land y \land z \land u \big\}.
	\end{align*}
	By Lemma~\ref{nec}, the set $T$ is closed under $C^\ast$. Let $(x, y, z) \in T$ be arbitrary with an element $u \in L$ witnessing that $(x,y,z) \in T$.  			From Lemma~\ref{pq} it follows that $p(x,y,z,u) = x \lor y \lor z \lor u$ if and only if $m \lor u = x \lor y \lor z$. Meeting both sides of the latter equality 			by $m'$, we get
	\begin{equation} \label{eq1}
		u \land m'= (m \land m') \lor (u \land m')= (m \lor u) \land m' = (x \lor y \lor z) \land m'.
	\end{equation}
	By the dual of Lemma~\ref{pq}, we know that $q(x,y,z,u) = x \land y \land z \land u$ if and only if $m \land u = x \land y \land z$. Then joining the last 			equality and \eqref{eq1}, we get that 
	\begin{align*}
	 u &= u \land 1_\L = u \land (m' \lor m) = (u \land m') \lor (u \land m)  \\ 
	& = ((x \lor y \lor z) \land m') \lor (x \land y \land z).
\end{align*}
It is not hard to derive from the defining identities of Boolean algebras that the latter formula is in fact the symmetric difference $x \vartriangle y \vartriangle z$ in $\B$. Alternatively, using Stone's representation theorem for Boolean algebras, we may assume that $x$, $y$ and $z$ are sets, and that the operations $\land, \lor, '$ are the set-theoretic intersection, union and complementation. Then $m$ corresponds to the set of elements that belong to at least two of the sets $x$, $y$ and $z$. Thus $(x \lor y \lor z) \land m'$ consists of those elements that belong to exactly one of $x$, $y$ and $z$, and $((x \lor y \lor z) \land m') \lor (x \land y \land z)$ contains those elements that belong to one or three of the sets $x$, $y$ and $z$, and this is indeed $x \vartriangle y \vartriangle z$ in $\B$.

	We have proved that the element $u$ witnessing that $(x,y,z)\in T$ can only be $x \vartriangle y \vartriangle z$:
	\begin{equation} \label{eq2}
		\forall x,y,z\in L\colon (x,y,z) \in T \iff \exists u \in L\colon u=x \vartriangle y \vartriangle z \iff x \vartriangle y \vartriangle z \in L.
	\end{equation}
	
	It is easy to see that $\{0_\L,1_\L\}^3 \subseteq T$, and using the main theorem of \cite{GG}, we get that if $(f, g) \in \Eq(T)$, then $f=g$ must hold. 			(In our case this theorem says that every term function of $\L$ is uniquely determined by its restriction to $\{0,1\}^3$.) Therefore only trivial equations 			can appear in $\Eq(T)$, hence $T=L^3$. Then \eqref{eq2} implies that $\L$ is closed under the ternary operation $x \vartriangle y \vartriangle z$. In particular, for any $x \in L$ we have 			$x \vartriangle 0 \vartriangle 1 = x' \in 	L$, which means that $\L$ is a Boolean lattice.
\end{proof}

We will need the following lemmas for the proof of Theorem~\ref{halo5}, which states that Boolean lattices have Property~$\sdc$. This will complete the determination of lattices with Property~$\sdc$.

\begin{lemma} \label{halo3}
	Let $\L=(L, \land, \lor)$ be a finite distributive lattice and let $C=[\land, \lor]\leq \O_L$. Then every system of $C$-equations is equivalent to a system of inequalities $\{p_1 \leq q_1, \dots, p_l\leq q_l\}$, where $p_i \in [\land]$ and $q_i \in [\lor]$ ($i=1,\dots, l$). 
\end{lemma}
\begin{proof}
	Let $\L=(L, \land, \lor)$ be a finite distributive lattice, let $C=[\land, \lor]\leq \O_L$ and let 
\[
\E=\{f_1=g_1, \dots, f_t=g_t\}
\]
	be a system of $C$-equations. For arbitrary $a, b \in L$ we have $a=b$ if and only if $a \leq b$ and $b \leq a$, therefore $\E$ is equivalent to the 			system of inequalities 
\[
\E'=\{ f_1 \leq g_1, g_1\leq f_1, \dots, f_t \leq g_t, g_t \leq f_t\}.
\]
	Denote the disjunctive normal forms of the left hand sides of the inequalities in $\E'$ as $DNF_j$, and denote the conjunctive normal forms of the right			hand sides of the inequalities in $\E'$ as $CNF_j$  ($j =1, \dots, 2t$). Then $\E'$ is equivalent to the system of inequalities 
\[
\{DNF_1 \leq CNF_1, \dots, DNF_{2t} \leq CNF_{2t}\}.
\]
	Each $DNF_j$ is a join of some meets, and each $CNF_j$ is a meet of some joins. Therefore, for every $j$ the inequality $DNF_j \leq CNF_j$ holds if and 	only if every meet in $DNF_j$ is less than or equal to every join in $CNF_j$. This means that there exists a system of inequalities $\{p_1 \leq q_1, \dots, p_l		\leq q_l\}$ equivalent to $\E$, such that $p_i \in [\land]$ and $q_i \in [\lor]$ ($i=1,\dots, l$). 
\end{proof}

\begin{lemma} \label{halo4}
	Let $\B=(B, \land, \lor, ')$ be a Boolean algebra. Then for every $a, b, c, d, u \in B$, we have
	\begin{enumerate}[label=(\roman*)]
		\item $a \land u \leq b \iff u \leq a' \lor b$;
		\item $b \leq a \lor u \iff u \geq a' \land b$;
		\item $a \land b' \leq c' \lor d \iff a \land c \leq b \lor d$.
	\end{enumerate}
\end{lemma}
\begin{proof}
	Let  $a, b, c, d, u \in B$ be arbitrary elements. For the proof of (i) let us first suppose that $a \land u \leq b$. Joining both sides of the inequality by 			$a'$, we get
\[
a' \lor (a \land u) = (a' \lor a) \land (a' \lor u) = 1_\B \land (a' \lor u) = a' \lor u \leq a' \lor b,
\]
	and from this, $u \leq a' \lor b$ follows.
	For the other direction, if $u \leq a' \lor b$ holds, then meeting both sides by $a$, we get that
\[
a \land u \leq a \land (a' \lor b) = (a \land a') \lor (a \land b) = 0_\B \lor (a \land b) = a \land b,
\]
	and from this, $a \land u \leq b$ follows.
	
	The second statement is the dual of (i).
	
	For the proof of (iii) let us use (i) with $u=a \land b'$, and then we get that
\[
a \land b' \leq c' \lor d \iff c \land (a \land b') = (c \land a) \land b' \leq d.
\]
	Then using (ii) with $u=d$, we get
\[
(c \land a) \land b' \leq d \iff c \land a \leq b \lor d,
\]
	which proves (iii).
\end{proof}

Helly's theorem from convex geometry states that if we have $k \ (>d)$ convex sets in $\mathbb{R}^d$, such that any $d+1$ of them have a nonempty intersection, then the intersection of all $k$ sets is nonempty as well. The following lemma says something similar for intervals in lattices (with $d=1$).

\begin{lemma} \label{helly}
	Let $\L=(L, \land, \lor)$ be a lattice, $c_i, d_i \in L$ ($i = 1, \dots, k$). Then we have
\[
\bigcap_{i=1}^k [c_i, d_i] \neq \emptyset \iff \forall i, j \in \{1, \dots, k\}\colon c_i \leq d_j.
\]
\end{lemma}
\begin{proof}
	Let $\L=(L, \land, \lor)$ be a lattice, and $c_i, d_i \in L$ ($i = 1, \dots, k$). Then obviously,
\[
\bigcap_{i=1}^k [c_i, d_i] = [c_1 \lor \dots \lor c_k, d_1 \land \dots \land d_k],
\]
	which is nonempty if and only if $c_1 \lor \dots \lor c_k \leq d_1 \land \dots \land d_k$, which holds if and only if $c_i \leq d_j$ for all $i,j \in \{1,...,k\}$.
\end{proof}

The last step in the characterization of finite lattices having Property~$\sdc$ is to show that Boolean lattices do indeed have Property~$\sdc$. For proving this, we will use the equivalence of this property with the quantifier\hyp{}eliminability for pp. formulas over $C^\circ =[\land, \lor]^\circ$ (see Theorem~\ref{equiv}).

\begin{theorem} \label{halo5}
	If $\L=(L, \land, \lor)$ is a finite Boolean lattice, then Property~$\sdc$ holds for $C=[\land, \lor]$.
\end{theorem}
\begin{proof}
	Let $\L=(L, \land, \lor)$ be a finite Boolean lattice, and let $C=[\land, \lor]$. Let us denote the complement of an element $x \in \L$ by $x'$. 
	By Theorem~\ref{equiv}, Property~$\sdc$ holds for $C$ if and only if any pp. formula over $C^\circ$ is equivalent to a quantifier-free pp. formula. Let us consider a pp. formula with a single quantifier:
	\begin{equation}
		\Phi(  x_{1},\dots,x_{n})  =\exists u\bigwith_{j=1}^{t}\rho_{j}\bigl(z_{1}^{(  j)  },\dots,z_{r_{j}}^{(  j)  }\bigr) 				\label{eqpp},
	\end{equation}
	where $\rho_j \in {(C^\circ})^{(r_j)}$, and $z_i^{(j)} \ (j = 1, \dots, t, \text{ and } i=1, \dots, r_j)$ are variables from the set $\{x_1, \dots, x_n, u\}$. 
	We will show that $\Phi$ is equivalent to a quantifier-free pp. formula, and thus (by iterating this argument) every pp. formula is equivalent to a quantifier-free pp. formula.
	By Lemma~\ref{halo3}, we can rewrite $\Phi$ to an equivalent formula 
\[
\exists u \bigwith_{i=1}^l (p_i \leq q_i),
\]
	where $p_i \in [\land]$ and $q_i \in [\lor]$ ($i = 1, \dots, l$).
	Let $a_i$ denote the meet of all variables from $\{x_1, \dots, x_n\}$ appearing in $p_i$, and let $b_i$ denote the join of all variables from $\{x_1, 			\dots, x_n\}$ appearing in $q_i$. Then we can distinguish four cases for the $i$-th inequality:
	\begin{enumerate}
	\setcounter{enumi}{-1}
		\item If $u$ does not appear in the inequality, then the inequality is of the form $a_i \leq b_i$.
		\item If $u$ appears only on the left hand side of the inequality, then the inequality is of the form $a_i \land u \leq b_i$.
		\item If $u$ appears only on the right hand side of the inequality, then the inequality is of the form $a_i \leq b_i \lor u$.
		\item If $u$ appears on both sides of the inequality, then the inequality is of the form $a_i \land u \leq b_i \lor u$, which always holds, since $a_i 					\land u \leq u \leq b_i \lor u$.
	\end{enumerate}
	
	Let $I_j$ denote the following set of indices:  
\[
I_j =\{i \mid \text{the inequality } p_i \leq q_i \text{ belongs to case (j)}\}
\]
	for $j=0, 1, 2, 3$. The only cases we have to investigate are case (1) and case (2) (since $u$ does not appear in case (0) and in case (3) there are only trivial inequalities).
	By Lemma~\ref{halo4}, 
	\begin{align*}
		& \text{for } i \in I_1 \text{ we have } a_i \land u \leq b_i \iff u \leq a_i' \lor b_i \iff u \in [0_\L, a_i' \lor b_i] =: [c_i, d_i]; \\
		&  \text{for } i \in I_2 \text{ we have } a_i \leq b_i \lor u \iff u \geq b_i' \land a_i \iff u \in [a_i \land b_i', 1_\L]=:[c_i, d_i].
	\end{align*}
	Then we have 
\[
\exists u \forall i \in I_1 \cup I_2\colon p_i \leq q_i \iff \bigcap_{i \in I_1 \cup I_2} [c_i, d_i] \neq \emptyset \iff \forall i, j \in I_1 \cup I_2\colon c_i \leq d_j
\]
	by Lemma~\ref{helly}. Since $u$ does not appear in the condition above, in principle, the quantifier has been eliminated. However, our formula still involves complements. Therefore, we use Lemma~\ref{halo4} to rewrite the formula. The only non-trivial case is if $c_i \neq 0_\L$ and $d_j \neq 1_\L$, that is, $c_i = a_i \land b_i'$ and $d_j= a_j' \lor b_j$ ($i \in I_2, j \in I_1$). In this case $c_i \leq d_j$ if and only if $a_i \land a_j \leq b_i \lor b_j$ by Lemma~\ref{halo4}.
	
	Summarizing the observations above, we have
	\begin{align*}
		\Phi(x_1, \dots, x_n) & \equiv \exists u \bigwith_{i=1}^l (p_i \leq q_i) \equiv \bigwith_{i \in I_0}(a_i \leq b_i) \bigwith_{i,j \in I_1 \cup I_2}(c_i \leq d_j) \\ & \equiv  \bigwith_{i \in I_0}(a_i \leq b_i) \bigwith_{i \in I_2,  j \in I_1}(a_i 	\land a_j \leq b_i \lor b_j), 
	\end{align*}
	which is equivalent to a quantifier-free pp. formula over $[\land, \lor]$ (since for all $x,y \in L$, we have $x \leq y$ if and only if $x = x \land y$).
\end{proof}

We can summarize the results of this section in the following theorem, which is a corollary of Lemma~\ref{halo1}, Theorem~\ref{halo2} and Theorem~\ref{halo5}.

\begin{theorem} \label{fohalo}
	A finite lattice has Property~$\sdc$ if and only if it is a Boolean lattice.
\end{theorem}

	This means that for any finite lattice $\L=(L, \land, \lor)$, solution sets of systems of equations over $\L$ can be characterized (via closure 	conditions) if and only if $\L$ is a Boolean lattice.

\section{Systems of equations over semilattices}\label{semil}

	Similarly to Section~\ref{latt}, in this section $\M=(M, \land)$ denotes a finite semilattice with meet operation $\land$ and least element $0_\M$.

\begin{lemma} \label{semil1}
	Let $\M=(M, \land)$ be a finite semilattice. If $\M$ has no greatest element, then Property~$\sdc$ does not hold for $C=[\land]$.
\end{lemma}
\begin{proof}
	Let $\M=(M, \land)$ be a finite semilattice with no greatest element, and let $C=[\land] \leq \O_M$. The set
	\begin{align*} 
		T & = \{(x,y) \mid \exists u \in L\colon x \land u = x \text{ and } y \land u = y \} = \\
		& = \{(x,y) \mid \exists u \in L\colon x \leq u \text{ and } y \leq u \}\subseteq M^2
	\end{align*}
	is closed under $C^\ast$ by Lemma~\ref{nec}. Similarly to Lemma~\ref{halo1}, we will prove that $T$ is not the solution set of any system of equations over $C$. 		Suppose that there exists a system of $C$-equations $\E$ such that $T=\Sol(\E)$. There are only three nontrivial 2-variable equations over $C$:
\[
 x=y,\quad x \land y = x,\quad x \land y = y.
\]
	 As in Lemma~\ref{halo1}, we prove that none of these equations can appear in $\E$ by presenting counterexamples to them (see Table~\ref{tab1}). Note that since $\M$ is finite and it has no greatest element, there exist maximal elements $a\neq b$ in $\M$. Then we have that only trivial equations can appear in $\E$, thus $T=M^2$. But this is a contradiction, since $(a, b) \notin T$.
\begin{table} \renewcommand{\arraystretch}{1.15}
	 \begin{tabular} 
			[c]{c|ccc}
			$=$ & $x$ & $y$ & $x \land y$\\\hline
			$x$ & $-$ & $(a,0_\M)$ & $(a,0_\M)$\\
			$y$ & $-$ & $-$ & $(0_\M,a)$\\
			$x \land y$ & $-$ & $-$ & $-$\\
	\end{tabular} \caption{Counterexamples showing that these equations do not belong to $\E$.} \label{tab1}
\end{table}
\end{proof}

	If a finite semilattice $\M=(M, \land)$ has a greatest element, then for all $(a, b) \in M^2$, the set $H=\{x \in M \mid a \leq x \text{ and } b \leq x\}$ is 		not empty. Since $\M$ is a finite semilattice, it follows that $\bigwedge H$ exists for all $(a, b) \in M^2$. This means that we can define a join 				operation $\lor$ on $M$, such that $\L=(L, \land, \lor)$ is a lattice (with $L=M$). Therefore, from now on it suffices to investigate lattices (but the clone 			we use for the equations is still $C=[\land]$).

The following theorem shows that Property~$\sdc$ does not hold for non-distributive lattices (regarded as semilattices), i.e., solution sets of systems of equations over a non-distributive lattice (as a semilattice) can not be characterized via closure conditions.

\begin{remark} \label{Gratzer}
 A meet semilattice $\M$ is \emph{distributive} if for any $a, b_0, b_1 \in \M$, the inequality $a \geq b_0 \land b_1$ implies that there exist $a_0, a_1 \in \M$ such that $a_0 \geq b_0$, $a_1 \geq b_1$ and $a=a_0 \land a_1$ (see Section 5.1 in Chapter II of \cite{GG2}).  From Lemma~184 of \cite{GG2} it follows that a finite semilattice is distributive if and only if it is a semilattice reduct of a distributive lattice.
\end{remark}

\begin{theorem} \label{semil2}
	Let $\L=(L, \land, \lor)$ be a finite lattice. If $\L$ is not distributive, then Property~$\sdc$ does not hold for $C=[\land]$.
\end{theorem}
\begin{proof}
	Let $\L=(L, \land, \lor)$ be a finite lattice and let $C=[\land]\leq \O_L$. Since $\L$ is not distributive, we know that there exists a sublattice of $\L$  isomorphic to either $N_5$ or $M_3$. Let us denote these two cases as ($N_5$) and ($M_3$), respectively. The figures and tables we use in this proof can be found in the Appendix. Let $T$ be the set
	\begin{align*}
		T & = \{(x,y,z) \in L^3 \mid \exists u \in L\colon x \land y = u \land y \text{ and } u \land x = x \text{ and } u \land z = z \} \\
		& =  \{(x,y,z) \in L^3 \mid \exists u \in L\colon x \land y = u \land y \text{ and } u \geq x \text{ and } u \geq z\},
	\end{align*}
	which is closed under $C^\ast$ by Lemma~\ref{nec}. As in Lemma~\ref{halo1}, we will prove that $T$ is not the solution set of any system of equations over $C$. 
	Similarly to Lemma~\ref{halo1}, we present counterexamples to nontrivial equations, the only difference is that here we prove that there can be only one nontrivial equation satisfied by $T$ (see tables~\ref{tabsl1} and \ref{tabsl2} for case ($N_5$) and ($M_3$), respectively).
	We choose the elements $a, b$ and $c$ as presented in Figure~\ref{figsl1} for case ($N_5$), and in Figure~\ref{figsl2} for case ($M_3$). (Note that an element $u$, chosen like on the figures, shows that in case ($N_5$) we have $(a, c, b), (b, a, c) \in T$, and in case ($M_3$) we have $(a, b, c), (a, c, b) \in T$.)

	So now we have that in both cases the only nontrivial equation that $T$ can satisfy is the equation $y \land z = x \land y \land z$. One can verify that this equation holds on $T$: if $(x,y,z) \in T$, then we have 
\[
 x \land y = u \land y \geq z \land y \implies x \land y \land z \geq y \land z,
\]
which implies that $y \land z = x \land y \land z$. Therefore, we can conclude that the only nontrivial equation in $\Eq(T)$ is $y \land z = x \land y \land z$. We will prove that $T$ is not the solution set of any system of equations by presenting a tuple $(x_1, y_1, z_1) \in  \Sol(\Eq(T)) \setminus T$ (cf. Remark~\ref{soleqt}). Since there exists a sublattice of $\L$ isomorphic to $N_5$ or $M_3$, there exists a tuple $(x_1, y_1, z_1)$ as shown in Figure~\ref{figsl3}, which satisfies $y_1 \land z_1 = x_1 \land y_1 \land z_1$, thus $(x_1, y_1, z_1) \in \Sol(\Eq(T))$. However, one can easily verify that $(x_1, y_1, z_1)$ does not belong to $T$. Indeed, suppose that $(x_1, y_1, z_1) \in T$, then there exists $u \in L$ such that $u \geq x_1$, 	$u \geq z_1$ and $x_1 \land y_1 = u \land y_1$. But then we have $u \geq x_1 \lor z_1 > y_1$ (since $N_5$ or $M_3$ is a sublattice), therefore $x_1 \land y_1 < u \land y_1 = y_1$ gives us a contradiction. Thus, $T \neq \Sol(\Eq(T))$, hence, by Remark~\ref{soleqt}, $T$ is not the solution set of any system of equations over $C$.
\end{proof}

Lemma~\ref{semil1} and Theorem~\ref{semil2} prove that if $\M = (M, \land)$ has Property~$\sdc$, then it is the semilattice reduct of a distributive lattice $\L = (L, \land, \lor)$. To complete the characterization of finite semilattices with Property~$\sdc$, we prove that the clone $[\land]$ has Property~$\sdc$ whenever $\land$ is the meet operation of a finite distributive lattice.

\begin{theorem} \label{semil4}
	If $\L=(L, \land, \lor)$ is a finite distributive lattice, then Property~$\sdc$ holds for $C=[\land]$.
\end{theorem}
\begin{proof}
	Let $\L=(L, \land, \lor)$ be a finite distributive lattice and $C=[\land]\leq \O_L$. Since $\L$ is distributive, by Birkhoff's representation theorem $\L$ can be 			embedded into a Boolean lattice $\B$, hence we may suppose without loss of generality that $\L$ is already a sublattice of $\B$. We can also assume 			that $0_\L=0_\B$ and $1_\L=1_\B$. Let us denote the complement of an element $x \in \B$ by $x'$. 
	By Theorem~\ref{equiv}, Property~$\sdc$ holds for $C$ if and only if any pp. formula over $C^\circ$ is equivalent to a quantifier-free pp. formula. Similarly to the proof of Theorem~\ref{halo5}, it suffices to consider pp. formulas with a single existential quantifier. Let 
	\begin{equation}
		\Phi(  x_{1},\dots,x_{n})  =\exists u\bigwith_{j=1}^{t}\rho_{j}\bigl(z_{1}^{(  j)  },\dots,z_{r_{j}}^{(  j)  }\bigr) 				\label{eqpp2},
	\end{equation}
	where $\rho_j \in {(C^\circ})^{(r_j)}$, and $z_i^{(j)} \ (j = 1, \dots, t, \text{ and } i=1, \dots, r_j)$ are variables from the set $\{x_1, \dots, x_n, u\}$. We will show that $\Phi$ is equivalent to a quantifier-free pp. formula.	
	
	Since for all $a, b \in L$ we have $a=b$ if and only if $a \leq b$ and $b \leq a$, we can rewrite $\Phi$ to an equivalent formula
\[
\exists u \bigwith_{i=1}^l (p_i \leq q_i),
\]
	where $p_i, q_i \in [\land]$ ($i=1, \dots, l$).
	Let $a_i$ denote the meet of all variables from $\{x_1, \dots, x_n\}$ appearing in $p_i$, and let $b_i$ denote the meet of all variables from $\{x_1, 			\dots, x_n\}$ appearing in $q_i$. Then we can distinguish four cases for the $i$-th inequality:
	\begin{enumerate}
	\setcounter{enumi}{-1}
		\item If $u$ does not appear in the inequality, then the inequality is of the form $a_i \leq b_i$.
		\item If $u$ appears only on the left hand side of the inequality, then the inequality is of the form $a_i \land u \leq b_i$.
		\item If $u$ appears only on the right hand side of the inequality, then the inequality is of the form $a_i \leq b_i \land u$, which holds if and only if 					$a_i \leq b_i$ and $a_i \leq u$.
		\item If $u$ appears on both sides of the inequality, then the inequality is of the form $a_i \land u \leq b_i \land u$, which holds if and only if $a_i 					\land u \leq b_i$ and $a_i \land u \leq u$, that is, $a_i \land u \leq b_i$.
	\end{enumerate}
	Let $I_j$ denote the following set of indices: 
\[
I_j =\{i \mid \text{the inequality } p_i \leq q_i \text{ belongs to case (j)}\}
\]
	for $j=0, 1, 2, 3$. We investigate only cases (1), (2) and (3), since $u$ does not appear in case (0). Moreover; in case (2), we only have to deal with the inequality $a_i \leq u$, since $u$ does not appear in the inequality $a_i \leq b_i$.
	By Lemma~\ref{halo4},
	\begin{align*}
		&  \text{for } i \in I_1 \text{ we have } a_i \land u \leq b_i \iff u \leq a_i' \lor b_i \iff u \in [0_\L, a_i' \lor b_i] =: [c_i, d_i]; \\
		&  \text{for } i \in I_2 \text{ we have } a_i \leq u \iff u \in [a_i, 1_\L]=:[c_i, d_i]; \\
		& \text{for } i \in I_3 \text{ we have } a_i \land u \leq b_i \iff u \leq a_i' \lor b_i \iff u \in [0_\L, a_i' \lor b_i] =: [c_i, d_i].
	\end{align*}
	Then we have 
\[
\bigcap_{i \in I_1 \cup I_2 \cup I_3} [c_i, d_i] \neq \emptyset \iff \forall i, j \in I_1 \cup I_2 \cup I_3\colon c_i \leq d_j
\]
	 by Lemma~\ref{helly}. Just as in the proof of Theorem~\ref{halo5}, we apply Lemma~\ref{halo4} to eliminate complements and joins from the formula above. The only interesting case is if $c_i \neq 0_\L$ and $d_j \neq 1_\L$, that is, $c_i = a_i$ and $d_j= a_j' \lor b_j$ ($i \in I_2, j \in I_1 \cup I_3$). In this case $c_i 			\leq d_j$ if and only if $a_i \leq a_j' \lor b_j$, which holds if and only if $a_i \land a_j \leq b_j$ by Lemma~\ref{halo4} (with $u=a_i$).
	
	Summarizing the observations above, we have
	\begin{align*}
		\Phi(x_1, \dots, x_n) & \equiv \exists u \bigwith_{i=1}^l (p_i \leq q_i) \equiv 
		\bigwith_{i \in I_0 \cup I_2}(a_i \leq b_i) \bigwith_{i,j \in I_1 \cup I_2 \cup I_3}(c_i \leq d_j) \\
		& \equiv \bigwith_{i \in I_0 \cup I_2}(a_i \leq b_i) \bigwith_{i \in I_2, j \in I_1 \cup I_3}(a_i \land a_j \leq  b_j),
	\end{align*}
	which is equivalent to a quantifier-free pp. formula over $[\land]$ (since for all $x,y \in L$, we have $x \leq y$ if and only if $x = x \land y$).
\end{proof}

We can summarize the results of this section in the following theorem, which is a corollary of Lemma~\ref{semil1}, and theorems \ref{semil2} and \ref{semil4}.

\begin{theorem} \label{fofelhalo}
	A finite semilattice has Property~$\sdc$ if and only if it is distributive.
\end{theorem}

	This means that for any finite semilattice $\M$, solution sets of systems of equations over $\M$ can be characterized (via closure conditions) if and only if $\M$ is a semilattice reduct of a distributive lattice (see Remark~\ref{Gratzer}).

\section{Concluding remarks}

We have characterized finite lattices and semilattices having Property~$\sdc$.
As a natural continuation of these investigations, one could aim at describing
all finite algebras (clones over finite sets) with Property~$\sdc$. 

Primitive positive clones seem to be of particular interest, for the following reason. For a primitive positive clone $P\leq\mathcal{O}_{A}$, let us consider the set $C(  P )  =\{  C\leq\mathcal{O}_{A}\colon C^{\ast}=P \}  $. The greatest element of this set is $P^{\ast}$, since
$C^{\ast}=P$ implies that $C\subseteq C^{\ast\ast}=P^{\ast}$ and $P^{\ast}\in C (  P )  $ follows from $P^{\ast\ast}=P$. If a clone $C\in C(P)$ has Property~$\sdc$, then every set $T\subseteq A^{n}$ that is
closed under $C^{\ast}=P$ arises as the solution set of a system $\mathcal{E}$
of equations over $C$. Since $C\subseteq P^{\ast}$, we can regard
$\mathcal{E}$ as a system of equations over $P^{\ast}$. Therefore, every set
$T\subseteq A^{n}$ that is closed under $(  P^{\ast})  ^{\ast}=P$
arises as the solution set of a system of equations over $P^{\ast}$, i.e.,
$P^{\ast}$ has Property~$\sdc$. Thus if $P^{\ast}$ does not satisfy
Property~$\sdc$, then no clone in $C(  P)  $ can have
Property~$\sdc$. In other words, primitive positive clones have the
\textquotedblleft highest chance\textquotedblright\ for having Property~$\sdc$.

Another topic worth further study is the relationship with ho\-mo\-mor\-phism-ho\-mo\-ge\-ne\-ity. It was proved in \cite{MasPec} that
ho\-mo\-mor\-phism-ho\-mo\-ge\-ne\-ity is equivalent to a certain quantifier elimination property (but somewhat different from Theorem~\ref{equiv}). Also, our results together with \cite{DolMas} imply that all finite lattices and semilattices with Property~$\sdc$ are ho\-mo\-mor\-phism-ho\-mo\-ge\-neous, so it might be plausible that Property~$\sdc$ implies ho\-mo\-mor\-phism-ho\-mo\-ge\-neity in general for finite algebras.

\section*{Acknowledgments}
The authors are grateful to Mikl\'{o}s Mar\'{o}ti, Dragan Ma\v{s}ulovi\'{c} and L\'aszl\'o Z\'adori for helpful discussions.

Research partially supported by the Hungarian Research, Development and Innovation Office under grants KH126581 and K128042 and by the Ministry of Human Capacities, Hungary grant 20391-3/2018/FEKUSTRAT.

Open Access. FundRef: University of Szeged Open Access Fund, Grant number: 4466.

\appendix
\section*{Appendix: figures and tables for the proof of Theorem~\ref{semil2}}
\begin{figure}[H]
		\begin{subfigure}{0.35\textwidth}
			\includegraphics[width=0.6\linewidth, center]{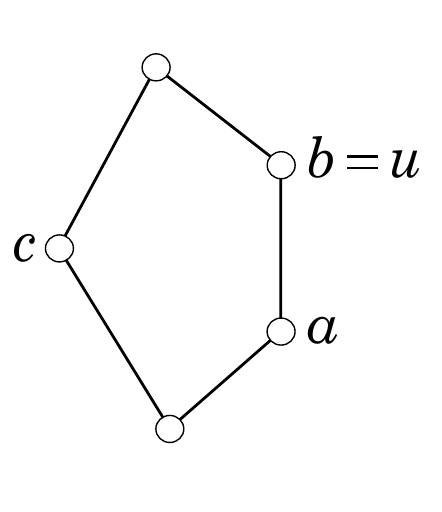}
		\end{subfigure}
		\begin{subfigure}{0.35\textwidth}
			\includegraphics[width=0.5\linewidth, center]{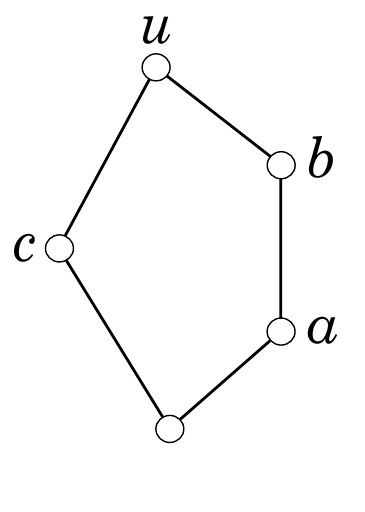}
		\end{subfigure}
		\caption{The elements $a, b$ and $c$ (with an example $u$ proving $(a, c, b), (b, a, c) \in T$) in case ($N_5$).} \label{figsl1}
\end{figure}

\begin{figure}[H]
		\begin{subfigure}{0.35\textwidth}
			\includegraphics[width=0.5\linewidth, center]{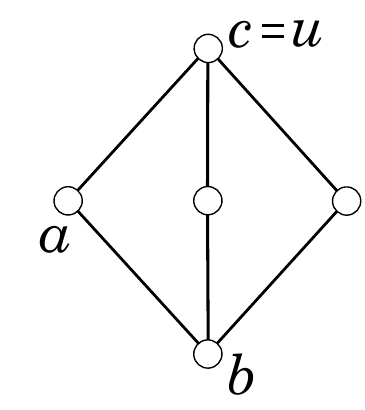}
		\end{subfigure}
		\begin{subfigure}{0.35\textwidth}
			\includegraphics[width=0.51\linewidth, center]{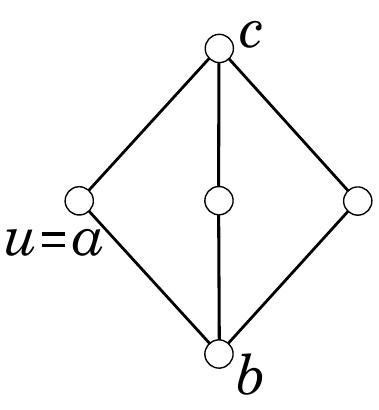}
		\end{subfigure}
		\caption{The elements $a, b$ and $c$ (with an example $u$ proving $(a, b, c), (a, c, b) \in T$) in case ($M_3$).} \label{figsl2}
\end{figure}

\begin{table}[H]
\[
\renewcommand{\arraystretch}{1.3}
\begin{tabular}
[c]{c|ccccccc}
$=$ & $x$ & $y$ & $z$ & $x \land y$ & $x \land z$ & $y \land z$ & $x \land y \land z$ \\\hline
$x$ & $-$ & $(a,c,b)$ & $(a,c,b)$ & $(a,c,b)$ & $(b,a,c)$ & $(a,c,b)$ & $(a,c,b)$\\
$y$ & $-$ & $-$ & $(a,c,b)$ & $(a,c,b)$ & $(a,c,b)$ & $(a,c,b)$ & $(a,c,b)$\\
$z$ & $-$ & $-$ & $-$ & $(a,c,b)$ & $(a,c,b)$ & $(a,c,b)$ & $(a,c,b)$\\
$x \land y$ & $-$ & $-$ & $-$ & $-$ & $(a,c,b)$ & $(b,a,c)$ & $(b,a,c)$\\
$x \land z$ & $-$ & $-$ & $-$ & $-$ & $-$ & $(a,c,b)$ & $(a,c,b)$\\
$y \land z$ & $-$ & $-$ & $-$ & $-$ & $-$ & $-$ & $ $\\
$x \land y \land z$ & $-$ & $-$ & $-$ & $-$ & $-$ & $-$ & $-$
\end{tabular}
\]
\caption{Counterexamples for case ($N_5$) showing that these equations do not belong to $\Eq(T)$.}  \label{tabsl1}
\end{table}

\begin{table}[H]
	\[
\renewcommand{\arraystretch}{1.3}
\begin{tabular}
[c]{c|ccccccc}
$=$ & $x$ & $y$ & $z$ & $x \land y$ & $x \land z$ & $y \land z$ & $x \land y \land z$ \\\hline
$x$ & $-$ & $(a,b,c)$ & $(a,b,c)$ & $(a,b,c)$ & $(a,c,b)$ & $(a,b,c)$ & $(a,b,c)$\\
$y$ & $-$ & $-$ & $(a,b,c)$ & $(a,c,b)$ & $(a,b,c)$ & $(a,c,b)$ & $(a,c,b)$\\
$z$ & $-$ & $-$ & $-$ & $(a,b,c)$ & $(a,b,c)$ & $(a,b,c)$ & $(a,b,c)$\\
$x \land y$ & $-$ & $-$ & $-$ & $-$ & $(a,b,c)$ & $(a,c,b)$ & $(a,c,b)$\\
$x \land z$ & $-$ & $-$ & $-$ & $-$ & $-$ & $(a,b,c)$ & $(a,b,c)$\\
$y \land z$ & $-$ & $-$ & $-$ & $-$ & $-$ & $-$ & $ $\\
$x \land y \land z$ & $-$ & $-$ & $-$ & $-$ & $-$ & $-$ & $-$
\end{tabular}
\]
\caption{Counterexamples for case ($M_3$) showing that these equations do not belong to $\Eq(T)$.}  \label{tabsl2}
\end{table}

\begin{figure}[H]
	\begin{subfigure}{0.35\textwidth}
		\includegraphics[width=0.5\textwidth, center]{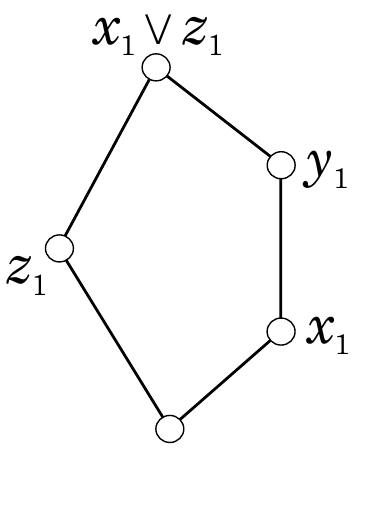}
	\end{subfigure}
	\begin{subfigure}{0.35\textwidth}
		\includegraphics[width=0.5\textwidth, center]{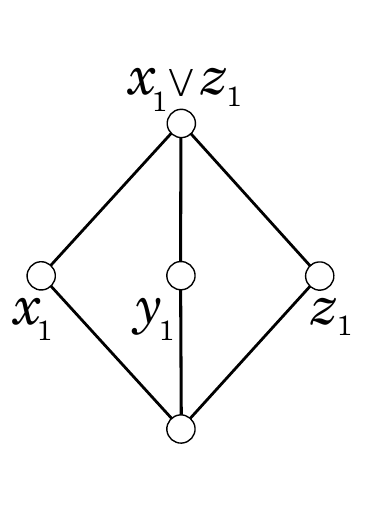}
	\end{subfigure}
		\caption{$(x_1, y_1, z_1)$ satisfies all equations in $\Eq(T)$, but $(x_1, y_1, z_1) \notin T$.}  \label{figsl3}
\end{figure}

\end{document}